\documentclass{amsart}
\usepackage{amsfonts}
\usepackage{amsmath}
\usepackage{amssymb}
\usepackage{amsthm}

\newtheorem{theorem}{Theorem}
\newtheorem{prop}{Proposition}
\newtheorem{lemma}{Lemma}
\newtheorem{rem}{Remark}

\newtheorem{cor}{Corollary}

\begin{document}

\title[Isometric embeddings of polar Grassmannians]{Isometric 
embeddings of polar Grassmannians and 
metric characterizations of their apartments}
\author{Mariusz Kwiatkowski, Mark Pankov}
\address{Department of Mathematics and Computer Science, University of Warmia and Mazury,
S{\l}oneczna 54, Olsztyn, Poland}
\email{mkv@matman.uwm.edu.pl, pankov@matman.uwm.edu.pl}

\maketitle

\begin{abstract}
We describe isometric embeddings of polar Grassmann graphs 
formed by non-maximal singular subspaces.
In almost all cases, they are induced by collinearity preserving injections 
of polar spaces.
As a simple consequence of this result, 
we get a metric characterization of apartments in polar Grassmannians.
\end{abstract}

\section{Introduction}
A {\it building} is a simplicial complex together with a distinguished family of subcomplexes,
so-called {\it apartments}, satisfying some axioms \cite{Tits}, 
see also \cite{BuekenhoutCohen-book}.
All apartments are identified with a certain Coxeter system which defines the building type.
The vertex set of the building can be labeled by the nodes of the diagram of 
this Coxeter system. The labeling is unique up to a diagram automorphism. 
The set of all vertices corresponding to the same node is a {\it building Grassmannian} 
\cite{Pankov-book1,Pasini}.
The intersections of apartments with Grassmannians
are said to be apartments in these Grassmannians. 
Grassmannians have the natural adjacency relation coming from the building structure:
two distinct vertices $a,b$ are {\it adjacent} if the building contains a simplex $P$ such that 
$P\cup\{a\}$ and  $P\cup\{b\}$ are chambers, i.e. maximal simplices in the building.
In this case, we say that the {\it line} joining $a$ and $b$  is the set of all 
vertices $c$ for which $P\cup\{c\}$ is a chamber.
So, every Grassmannian can be considered as a graph as well as a point-line geometry.
 
Every building of type $\textsf{A}_{n}$, $n\ge 3$
is the flag complex of a certain $(n+1)$-dimen\-sional vector space over a division ring
and the corresponding Grassmannians are formed by subspaces of the same dimension.
Similarly, every building of type $\textsf{C}_{n}$ is the flag complex of a rank $n$ polar space
and all buildings of type $\textsf{D}_{n}$ can be obtained from polar spaces of type 
$\textsf{D}_{n}$. 
The Grassmannians of such buildings are {\it polar} and {\it half-spin Grassmannians}.
The associated graphs are said to be {\it polar} and {\it half-spin Grassmann graphs}.
Note that the description of all automorphisms of these graphs \cite[Section 4.6]{Pankov-book1}
is a generalization of classical Chow's theorems \cite{Chow}.
In this paper isometric embeddings of polar Grassmann graphs will be considered.
In what follows we denote by $\Gamma_{k}(\Pi)$
the polar Grassmann graph formed by $k$-dimensional singular subspaces of 
a polar space $\Pi$.

Let $\Pi$ and $\Pi'$ be polar spaces of rank $n$ and $n'$, respectively.
By \cite[Theorem 3]{Pankov1}, every isometric embedding of 
the dual polar graph $\Gamma_{n-1}(\Pi)$ in the dual polar graph $\Gamma_{n'-1}(\Pi')$
is induced by a collinearity preserving injection of $\Pi$ to
the quotient polar space of $\Pi'$ by a certain $(n'-n-1)$-dimensional singular subspace.
It follows from \cite[Theorem 2]{Pankov1} that apartments in 
the polar Grassmannian formed by maximal singular subspaces of $\Pi$
can be characterized as the images of isometric embeddings of 
the $n$-dimensional hypercube graph $H_{n}$ in $\Gamma_{n-1}(\Pi)$.
If $\Pi$ and $\Pi'$ are polar spaces of types $\textsf{D}_{n}$ and $\textsf{D}_{n'}$
(respectively) and $n$ is even then 
the same holds for isometric embeddings of 
the associated half-spin Grassmann graphs \cite[Theorem 4]{Pankov3}.
By \cite[Theorem 2]{Pankov3}, apartments in the half-spin Grassmannians of $\Pi$
can be characterized as the images of isometric embeddings of 
the half-cube graph $\frac{1}{2}H_{n}$ in the corresponding half-spin Grassmann graphs;
as above, we assume that $n$ is even.
Also, there is the following conjecture \cite[Section 6]{Pankov3}:
if $n$ is odd then there exist isometric embeddings of $\frac{1}{2}H_{n}$ in 
the half-spin Grassmann graphs of $\Pi$ whose images are not apartments.

In this paper similar results will be established for isometric embeddings 
of polar Grassmann graphs formed by non-maximal singular subspaces
(Theorems \ref{theorem-main1}--\ref{theorem-main3}).
Our arguments are different from the arguments given in \cite{Pankov1,Pankov3}.
In dual polar graphs and half-spin Grassmann graphs 
the distance between two vertices is completely defined by 
the dimension of the intersection of the corresponding maximal singular subspaces.
For polar Grassmann graphs formed by non-maximal singular subspaces 
the distance formula is more complicated (Subsection 2.4).

As a simple consequence of the main results, 
we get the following metric characterization of apartments in polar Grassmannians
(Corollary 1): 
if $\Gamma_{k}(n)$ denotes the restriction of the graph $\Gamma_{k}(\Pi)$ to any apartment 
then the image of every isometric embedding of $\Gamma_{k}(n)$ in 
$\Gamma_{k}(\Pi)$ is an apartment.

It must be pointed out that there is no similar characterization
for apartments in Grassmannians of vector spaces.
Let $V$ be an $n$-dimensional vector space (over a division ring).
Consider the Grassmann graph $\Gamma_{k}(V)$  
formed by $k$-dimensional subspaces of $V$.
The restriction of $\Gamma_{k}(V)$ to every apartment of 
the corresponding Grassmannian is isomorphic to the Johnson graph $J(n,k)$.
The image of every isometric embedding of $J(n,k)$ in $\Gamma_{k}(V)$
is an apartment only in the case when $n=2k$.
The images of all possible isometric embeddings of Johnson graphs in Grassmann graphs 
are described in \cite[Chapter 4]{Pankov-book2}.
Also, \cite[Chapter 3]{Pankov-book2} contains the complete description of 
isometric embeddings of Grassmann graphs.
They are defined by semilinear embeddings of special type
and are more complicated than isometric embeddings of polar Grassmann graphs.

Other characterizations of apartments in building Grassmannians 
can be found in \cite{CKS,Coop-sur,Kasikova1,Kasikova2,Kasikova3,Pankov2}.
Some of them are in terms of independent subsets of point-line geometries.
Note that building Grassmannians can be contained in other building Grassmannians
as subspaces (in the sense of point-line geometry).
Is it possible to determine all such subspaces?
This problem is closely related to characterizing of apartments  and 
solved for some special cases \cite{BC,CKS,Coop1,Coop2}.
For example, subspaces  of polar Grassmannians isomorphic to 
Grassmannians of vector spaces are described in \cite{BC}.
There is a similar description for subspaces of symplectic Grassmannians 
isomorphic to other symplectic Grassmannians \cite{Coop2}.

\section{Basic notions and constructions}

\subsection{Graphs}
We define a {\it graph} as a pair $\Gamma=(X,\sim)$,
where $X$ is a non-empty set (possibly infinite) whose elements are called {\it vertices} 
and $\sim$ is a symmetric relation on $X$ called {\it adjacency}.
We say that vertices $x,y\in X$ are {\it adjacent} if $x\sim y$.
Every pair of adjacent vertices form an {\it edge}.
We suppose that $x\not\sim x$ for every $x\in X$,
i.e. our graph does not contain loops.
A  {\it clique} is a subset of $X$,  
where any two distinct elements are adjacent vertices of $\Gamma$.
Using Zorn lemma, we show that every clique is contained in a certain maximal clique.

We will consider connected graphs only.
In such a graph we define the {\it distance} $d(x,y)$ between two vertices $x,y$ 
as the smallest number $i$ such that there is a path consisting of $i$ edges and 
connecting $x$ and $y$ \cite[Section 15.1]{DD}.
A path between $x$ and $y$ is said to be a {\it geodesic} if 
it is formed by precisely $d(x,y)$ edges. 
The graph {\it diameter} is the greatest distance between two vertices.

An {\it embedding} of a graph $\Gamma$ in a graph $\Gamma'$
is an injection of the vertex set of $\Gamma$ to the vertex set of $\Gamma'$
transferring adjacent and non-adjacent vertices of $\Gamma$ to
adjacent and non-adjacent vertices of $\Gamma'$, respectively. 
Surjective embeddings are isomorphisms.
Every embedding $f$ sends maximal cliques of $\Gamma$ to cliques of $\Gamma'$
which are not necessarily maximal, i.e. subsets of maximal cliques.
For any distinct  maximal cliques ${\mathcal X}$ and ${\mathcal Y}$ of $\Gamma$
there exist non-adjacent vertices $x\in {\mathcal X}$ and $y\in {\mathcal Y}$.
Then $f(x)$ and $f(y)$ are non-adjacent vertices of $\Gamma'$
and there is no clique containing both $f({\mathcal X})$ and $f({\mathcal Y})$.
So, every embedding transfers  distinct maximal cliques to subsets of distinct maximal cliques.

An embedding is {\it isometric} if it preserves the distance between vertices.

\subsection{Polar spaces}
A {\it partial linear space} is a pair $\Pi=(P,{\mathcal L})$, 
where $P$ is a non-empty set whose elements are called {\it points}
and ${\mathcal L}$ is a family of proper subsets of $P$ called {\it lines}.
Every line contains at least two points and every point belongs to a certain  line.
Also, for any two distinct points there is at most one line containing them.
The points are said to be {\it collinear} if such a line exists.
A {\it subspace} of $\Pi$ is a subset $S\subset P$ such that for 
any two collinear points of $S$ the line joining them is contained in $S$.
A subspace is called {\it singular} if any two distinct points of this subspace are collinear.
The empty set, one-point sets and lines are singular subspaces.
Using Zorn lemma, we establish that 
every singular subspace is contained in a maximal singular subspace.

By \cite{BuekenhoutCohen-book,Pankov-book1,Shult-book,Ueberberg}, 
a {\it polar space} is a partial linear space satisfying the following axioms:
\begin{enumerate}
\item[(P1)] every line contains at least three points,
\item[(P2)] there is no point collinear to all points,
\item[(P3)] for every point and every line
the point is collinear to one or all points of the line,
\item[(P4)] any chain of mutually distinct incident singular subspaces is finite.
\end{enumerate}
If a polar space has a singular subspace containing more than one line
then all maximal singular subspaces are projective spaces
of the same dimension $n\ge 2$ and the number $n+1$ is called the {\it rank} of this polar space.
Polar spaces of rank $2$ (all maximal singular subspaces are lines) are known as 
{\it generalized quadrangles}. 
In the case when the rank of a polar is greater than $2$,
every singular subspace is a subspace of a certain projective space 
and its dimension is well-defined.

Two polar spaces $\Pi=(P,{\mathcal L})$ and $\Pi'=(P',{\mathcal L}')$ are {\it isomorphic} 
if there is a {\it collineation} of $\Pi$ to $\Pi'$, 
i.e. a bijection $\alpha:P\to P'$ such that $\alpha({\mathcal L})={\mathcal L}'$.

All polar spaces of rank $\ge 3$ are known \cite{Tits}.
For example, there are polar spaces related to 
non-degenerate reflexive forms (alternating, symmetric and hermitian).
If such a form is trace-valued and has isotropic subspaces of dimension at least $2$
then it defines a polar space: 
the point set is formed by all $1$-dimensional isotropic subspaces,
the lines are defined by $2$-dimensional isotropic subspaces 
and other isotropic subspaces correspond to singular subspaces of dimension greater than $1$.

Consider the $(2n)$-element set $J:=\{\pm 1,\dots, \pm n\}$
and the partial linear space $\Pi_{n}$
whose point set is $J$ and whose lines are $2$-element subsets $\{i,j\}$ such that $j\ne -i$.
Then $S\subset J$ is a singular subspace of $\Pi_{n}$ 
if and only if for every $i\in S$ we have $-i\not\in S$.
A singular subspace is maximal if it consists of $n$ points.
The dimension of a singular subspace $S$ is equal to $|S|-1$ and 
maximal singular subspaces of $\Pi_{n}$ are $(n-1)$-dimensional.
The partial linear space $\Pi_{n}$ satisfies the axioms (P2)--(P4) 
and we say that every partial linear space isomorphic to $\Pi_{n}$ 
is a {\it thin polar space of rank $n$}.

Let $\Pi=(P,{\mathcal L})$ be a polar space of rank $n$.
For every subset $X\subset P$ the subspace of $\Pi$ spanned by $X$,
i.e. the minimal subspace containing $X$, is denoted by $\langle X \rangle$.
If any two distinct points of $X$ are collinear then this subspace is singular.
If a point is collinear to every point of $X$
then this point is collinear to all points of the subspace $\langle X\rangle$.
A subset of $P$ consisting of $2n$ distinct points $p_{1},\dots,p_{2n}$
is a {\it frame} of $\Pi$ if for every $i$ there is unique $\sigma(i)$
such that $p_{i}$ and $p_{\sigma(i)}$ are non-collinear. 
Any $k$ distinct mutually collinear points in a frame span 
a $(k-1)$-dimensional singular subspace. 
We will use the following remarkable property of frames:
for any two singular subspaces there is a frame such that 
these subspaces are spanned by subsets of the frame.
Note that a thin polar space
contains the unique frame which coincides with the set of points.

Every rank $n$ polar space satisfies one the following conditions:
\begin{enumerate}
\item[($\textsf{C}_{n}$)] every $(n-2)$-dimensional singular subspace 
is contained in at least three maximal singular subspaces,
\item[($\textsf{D}_{n}$)] every $(n-2)$-dimensional singular subspace 
is contained in precisely two maximal singular subspaces.
\end{enumerate}
We say that a polar space is of type $\textsf{C}_{n}$ or $\textsf{D}_{n}$
if the corresponding possibility is realized.
For example, if a rank $n$ polar space is defined by an alternating or hermitian form 
then it is of type $\textsf{C}_{n}$. 
A thin polar space of rank $n$ is of type $\textsf{D}_{n}$.
Other  polar spaces  of this type will be considered in Subsection 2.5.

\subsection{Polar Grassmannians}
Let $\Pi=(P, {\mathcal L})$ be a polar space or a thin polar space of rank $n$.
For every $k\in \{0,1,\dots,n-1\}$ we denote by ${\mathcal G}_{k}(\Pi)$
the polar Grassmannian consisting of $k$-dimensional singular subspaces of $\Pi$.
Note that ${\mathcal G}_{0}(\Pi)$ coincides with $P$
and ${\mathcal G}_{n-1}(\Pi)$ is formed by maximal singular subspaces.

The {\it polar Grassmann graph $\Gamma_{k}(\Pi)$}
is the graph whose vertex set is ${\mathcal G}_{k}(\Pi)$.
In the case when  $k\le n-2$, two distinct elements of ${\mathcal G}_{k}(\Pi)$
are adjacent vertices of $\Gamma_{k}(\Pi)$
if there is a $(k+1)$-dimensional singular subspace containing them.
Two distinct maximal singular subspaces are adjacent vertices of $\Gamma_{n-1}(\Pi)$
if their intersection is $(n-2)$-dimensional. 
The graph $\Gamma_{n-1}(\Pi)$ is known as the {\it dual polar graph} associated to $\Pi$.
If $\Pi$ is a thin polar space then
we write $\Gamma_{k}(n)$ instead of $\Gamma_{k}(\Pi)$.
Note that $\Gamma_{n-1}(n)$ is isomorphic to the $n$-dimensional hypercube graph $H_{n}$.

For every frame of $\Pi$ the set consisting of all $k$-dimensional 
singular subspaces spanned by subsets of the frame is called
the {\it apartment} of ${\mathcal G}_{k}(\Pi)$ associated to this frame.
The restriction of the graph $\Gamma_{k}(\Pi)$ to every apartment of  
${\mathcal G}_{k}(\Pi)$ is isomorphic to $\Gamma_{k}(n)$.
By the frame property given in the previous subsection,
for any two elements of ${\mathcal G}_{k}(\Pi)$ there is an apartment containing them.
If $\Pi$ is a thin polar space then
there is the unique apartment of ${\mathcal G}_{k}(\Pi)$ which coincides with the 
polar Grassmannian. 

For every singular subspace $S$ 
we denote by $[S\rangle_{k}$ the set of all $k$-dimensional singular subspaces
containing $S$. 
This set is non-empty only in the case when the dimension of $S$ is not greater than $k$.
Every subset of type 
$$[S\rangle_{n-1},\;\;S\in {\mathcal G}_{n-2}(\Pi)$$
is called a {\it line} of ${\mathcal G}_{n-1}(\Pi)$. 
Each maximal clique in the dual polar graph $\Gamma_{n-1}(\Pi)$ is a line. 

Now we suppose that $k\le n-2$.
Let $S$ and $U$ be a pair of incident singular subspaces such that 
$\dim S \le k\le \dim U$.
Denote by $[S,U]_{k}$ the set of all $X\in {\mathcal G}_{k}(\Pi)$
satisfying $S\subset X\subset U$.
In the case when $S=\emptyset$, we write $\langle U]_{k}$ instead of $[S,U]_{k}$.
If 
$$\dim S=k-1\;\mbox{ and }\;\dim U=k+1$$
then $[S,U]_{k}$ is called a {\it line} of ${\mathcal G}_{k}(\Pi)$.
In the case when $k=0$, we get a line of $\Pi$.

If $1\le k\le n-3$ then there are precisely the following two types of maximal cliques of 
$\Gamma_{k}(\Pi)$:
\begin{enumerate}
\item[$\bullet$] the {\it star} $[S,U]_{k}$, 
$S\in {\mathcal G}_{k-1}(\Pi)$ and $U\in {\mathcal G}_{n-1}(\Pi)$;
\item[$\bullet$] the {\it top} $\langle U]_{k}$,  
$U\in {\mathcal G}_{k+1}(\Pi)$.
\end{enumerate}
Every star of ${\mathcal G}_{n-2}(\Pi)$ is a line contained in a certain top and
all maximal cliques of $\Gamma_{n-2}(\Pi)$ are tops. 
Tops and stars of ${\mathcal G}_{0}(\Pi)=P$ are 
lines and maximal singular subspaces of $\Pi$, respectively.

In the case when $1\le k \le n-2$,
every subset of type 
$$[S\rangle_{k},\;\;S\in {\mathcal G}_{k-1}(\Pi)$$
is said to be a {\it big star}.
Every big star $[S\rangle_{k}$ 
(together with all lines of ${\mathcal G}_{k}(\Pi)$ contained in it)
is a polar space of rank $n-k$.
We denote this polar space by $\Pi_{S}$.
Every $i$-dimensional singular subspace of $\Pi_{S}$ is a subset of type 
$[S,U]_{k}$,
where $U$ is a $(k+i)$-dimensional singular subspace containing $S$.
Therefore, for every $i\in \{0,\dots,n-k-1\}$
the Grassmannian ${\mathcal G}_{i}(\Pi_{S})$ can be naturally identified with 
the set $[S\rangle_{k+i}$ and the polar Grassmann graph 
$\Gamma_{i}(\Pi_{S})$ coincides with the restriction of the graph $\Gamma_{k+i}(\Pi)$
to $[S\rangle_{k+i}$.
If ${\mathcal A}$ is an apartment of ${\mathcal G}_{k}(\Pi)$
such that $S$ is spanned by a subset of the frame associated to ${\mathcal A}$
then ${\mathcal A}\cap [S\rangle_{k}$ is a frame of $\Pi_{S}$.
Conversely, every frame of $\Pi_{S}$ can be obtained in this way.
Similarly, every apartment of ${\mathcal G}_{i}(\Pi_{S})$
is the intersection of $[S\rangle_{k+i}$ and an apartment of ${\mathcal G}_{k+i}(\Pi)$
such that $S$ is spanned by a subset of the associated frame.

\subsection{Distance in polar Grassmann graphs}
The Grassmann graph $\Gamma_{k}(\Pi)$ is connected for every $k$.
The distance between $X,Y\in {\mathcal G}_{n-1}(\Pi)$ in 
the dual polar graph $\Gamma_{n-1}(\Pi)$ is 
$$n-1-\dim(X\cap Y).$$
In particular, the diameter of $\Gamma_{n-1}(\Pi)$ is equal to $n$
(the dimension of the empty set is $-1$).

Let $X,Y\in  {\mathcal G}_{k}(\Pi)$ and $k\le n-2$.
Suppose that there is a point $p\in X\setminus Y$ collinear to all points of $Y$. 
Then there exists a point $q\in Y\setminus X$ collinear to all points of $X$.
This follows, for example, from the existence of a frame of $\Pi$ containing 
the point $p$ and
such that $X$ and $Y$ are spanned by subsets of this frame.
If $X\cap Y$ is $(k-1)$-dimensional then $X$ and $Y$ are adjacent vertices of 
$\Gamma_{k}(\Pi)$.
In the case when 
$$\dim (X\cap Y)\le k-2,$$ 
we take any $k$-dimensional singular subspace $X'$
spanned by the point $q$ and a $(k-1)$-dimensional subspace of $X$
containing $X\cap Y$ and $p$. 
Then $X$ and $X'$ are adjacent vertices of $\Gamma_{k}(\Pi)$
and 
$$\dim (X'\cap Y)=\dim (X\cap Y)+1.$$
Note that $p$ is a point of $X'\setminus Y$ collinear to all points of $Y$.
Using induction we show that 
$$d(X,Y)=k-\dim (X\cap Y).$$
Now we suppose that every point of $X\setminus Y$ is non-collinear to a certain point of $Y$.
Then every point of $Y\setminus X$ is non-collinear to a certain point of $X$. 
We take any frame whose subsets span $X$ and $Y$
and construct a $k$-dimensional singular subspace $X'$ satisfying the following conditions:
\begin{enumerate}
\item[(1)] $X$ and $X'$ are adjacent vertices of $\Gamma_{k}(\Pi)$,
\item[(2)] there is a point of $X'\setminus Y$ collinear to all points of $Y$,
\item[(3)] $X\cap Y=X'\cap Y$. 
\end{enumerate}
Then 
$$d(X',Y)=k-\dim(X'\cap Y)=k-\dim(X\cap Y).$$
We have
$$d(X,Y)=k-\dim(X\cap Y)+1,$$
since there is no vertex of $\Gamma_{k}(\Pi)$ adjacent to $X$
and intersecting $Y$ in a subspace of dimension greater than 
the dimension of $X\cap Y$.

So,  if $k\le n-2$ then the diameter of $\Gamma_{k}(\Pi)$ is equal to $k+2$
and we have the following description of the distance.

\begin{lemma}\label{lemma-dist}
Let $X,Y\in {\mathcal G}_{k}(\Pi)$ and $k\le n-2$.
If the distance between $X$ and $Y$ in $\Gamma_{k}(\Pi)$
is equal to $m$ then one of the following possibilities is realized:
\begin{enumerate}
\item[{\rm (1)}] $\dim(X\cap Y)=k-m$, 
there is a point of $X\setminus Y$ collinear to all points of $Y$
and there is a point of $Y\setminus X$ collinear to all points of $X$;
\item[{\rm (2)}] $m>1$, $\dim(X\cap Y)=k-m+1$,
every point of $X\setminus Y$ is non-collinear to a certain point of $Y$
and every point of $Y\setminus X$ is non-collinear to a certain point of $X$.
\end{enumerate}
If $m=k+2$ then only the second possibility is realized.
\end{lemma}

\subsection{Polar spaces of type $\textsf{D}_{n}$ and half-spin Grassmannians}
It was noted above that a thin polar space of rank $n$ is of type $\textsf{D}_{n}$.
Let $V$ be a $(2n)$-dimensional vector space over a field. 
If the characteristic of this field is not equal to 2 and 
there is a non-degenerate symmetric bilinear form on $V$ 
whose maximal isotropic subspaces are $n$-dimensional 
then the associated polar space is of type $\textsf{D}_n$. 
In the case when the characteristic of the field is equal to 2, we consider a
non-defect quadratic form on $V$ such that maximal singular subspaces are $n$-dimensional.
The associated polar space (the points are $1$-dimensional singular subspaces and the lines
are defined by $2$-dimensional singular subspaces) is also of type $\textsf{D}_n$. 
It follows from Tits's description of polar spaces \cite{Tits} 
that every polar space of type $\textsf{D}_n$, $n\ge 4$ 
is isomorphic to one of the polar spaces mentioned above.

Let $\Pi=(P, {\mathcal L})$ be a polar space of type $\textsf{D}_n$ (possibly thin)
and $n\ge 4$.
Then ${\mathcal G}_{n-1}(\Pi)$ can be uniquely decomposed
in the sum of two disjoint subsets ${\mathcal G}_{+}(\Pi)$ and ${\mathcal G}_{-}(\Pi)$
such that the distance between any two elements of 
${\mathcal G}_{\delta}(\Pi)$, $\delta\in\{+, -\}$ in the dual polar graph $\Gamma_{n-1}(\Pi)$
is even and the same distance between any 
$S\in {\mathcal G}_{+}(\Pi)$ and $U\in {\mathcal G}_{-}(\Pi)$ is odd.
These subsets are known as the {\it half-spin Grassmannians}.

If $\Pi$ is defined by a non-degenerate symmetric bilinear form $\Omega$ 
then the maximal singular subspaces of $\Pi$ are identified with the
maximal isotropic subspaces of $\Omega$ and 
the half-spin Grassmannians are the orbits
of the action of the orthogonal group $\rm{O}_{+}(\Omega)$ 
on the set of all maximal isotropic subspaces. 
Every element of $\rm{O}(\Omega)\setminus \rm{O}_{+}(\Omega)$ 
induces a collineation of $\Pi$  which transfers one of the half-spin Grassmannians to 
the other. 
The same holds for the case when $\Pi$ is defined by a quadratic form.
So, collineations of $\Pi$ sending ${\mathcal G}_{+}(\Pi)$ to ${\mathcal G}_{-}(\Pi)$
always exist.

Suppose that $n=4$ and $\delta\in\{+, -\}$.
For every line $L\in {\mathcal L}$ the set $[L\rangle_{\delta}$
consisting of all elements of ${\mathcal G}_{\delta}(\Pi)$ containing $L$
is called a {\it line} of ${\mathcal G}_{\delta}(\Pi)$. 
The half-spin Grassmannian ${\mathcal G}_{\delta}(\Pi)$ together
with the family of all such lines is a polar space of type $\textsf{D}_{4}$.
We denote this polar space by $\Pi_{\delta}$.
The polar spaces $\Pi_{+}$ and $\Pi_{-}$ are isomorphic
(every collineation of $\Pi$ transferring ${\mathcal G}_{+}(\Pi)$ to ${\mathcal G}_{-}(\Pi)$
induces a collineation between these polar spaces).
The half-spin Grassmannians corresponding to $\Pi_{\delta}$ are the point set $P$
and ${\mathcal G}_{-\delta}(\Pi)$, 
where $-\delta$ is the complement of $\delta$ in the set $\{+,-\}$. 
The associated polar spaces are $\Pi$ and $\Pi_{-\delta}$.
Therefore, $\Pi$ is isomorphic to both $\Pi_{+}$ and $\Pi_{-}$.
Since there is a natural one-to-one correspondence between lines of 
the polar spaces $\Pi$ and $\Pi_{\delta}$, every collineation of $\Pi$ to $\Pi_{\delta}$
induces a bijective transformation of ${\mathcal G}_{1}(\Pi)$.
This transformation is an automorphism of the graph $\Gamma_{1}(\Pi)$.

Let $\alpha$ be a collineation of $\Pi$ to $\Pi_{-\delta}$.
It induces collineations of $\Pi_{+}$ and $\Pi_{-}$
to the polar spaces associated to the half-spin Grassmannians of $\Pi_{-\delta}$.
So, we get a collineation of $\Pi_{\delta}$ to $\Pi$ or $\Pi_{\delta}$.
Since there are collineations of $\Pi_{-\delta}$ transferring $\Pi$ to $\Pi_{\delta}$,
we can suppose that $\alpha$ induces a collineation of $\Pi_{\delta}$ to itself.
The automorphism $g$ of $\Gamma_{1}(\Pi)$ induced by $\alpha$
has the following properties:
\begin{enumerate}
\item[$\bullet$] if $U\in {\mathcal G}_{\delta}(\Pi)$ 
then $g([p,U]_{1})$ is a top for every $p\in U$,
\item[$\bullet$]  if $U\in {\mathcal G}_{-\delta}(\Pi)$
then there exists  $U'\in {\mathcal G}_{-\delta}(\Pi)$ such that 
every $g([p,U]_{1})$ is a star contained in $\langle U']_{1}$.
\end{enumerate} 
See \cite[Section 4.6]{Pankov-book1} for the details.

\section{Main results}
From this moment we suppose that 
$\Pi=(P,{\mathcal L})$ is a polar space or a thin polar space of rank $n$
and $\Pi'=(P',{\mathcal L}')$ is a polar space of rank $n'$.

Let $f:P\to P'$ be a collinearity preserving injection,
i.e. $f$ sends collinear and non-collinear points of $\Pi$
to collinear and non-collinear points of $\Pi'$, respectively.
Show that $f$ transfers every frame of $\Pi$ to a subset in a frame of $\Pi'$.

If ${\mathcal F}$ is a frame of $\Pi$ then 
for every point $p\in f({\mathcal F})$ there is a unique point of $f({\mathcal F})$
non-collinear to $p$. 
This means that $n\le n'$ and $f({\mathcal F})$ is a frame of $\Pi'$ if $n=n'$.
In the case when $n'>n$, 
we consider the set formed by all points of $\Pi'$ collinear to all points of $f({\mathcal F})$.
If $n'-n\ge 2$ then this is a polar space of rank $n'-n$ and
$f({\mathcal F})$ together with any frame of this polar space give a frame of $\Pi'$.
If $n'-n=1$ then our set consists of mutually non-collinear points and
$f({\mathcal F})$ together with any pair of such points define a frame of $\Pi'$.

Since every singular subspace $S$ of $\Pi$ 
is spanned by a subset of a certain frame of $\Pi$, 
the dimension of the singular subspace $\langle f(S) \rangle$ is equal to the dimension of $S$. 
It is clear that $f$ is an isometric embedding of $\Gamma_{0}(\Pi)$ in $\Gamma_{0}(\Pi')$
and for every $k\in \{1,\dots,n-1\}$ the mapping 
$$(f)_{k}:{\mathcal G}_{k}(\Pi)\to {\mathcal G}_{k}(\Pi')$$
$$S\to \langle f(S) \rangle$$
is an isometric embedding of $\Gamma_{k}(\Pi)$ in $\Gamma_{k}(\Pi')$.
If $n=n'$ and $\Pi$ is a thin polar space then the image of this mapping
is an apartment of ${\mathcal G}_{k}(\Pi')$.

Now we suppose that $n\le n'$ and take any 
$m$-dimensional singular subspace $S$ of $\Pi'$ such that 
$$m\le n'-n-1$$
(this subspace is empty if $n=n'$).
Then $\Pi'_{S}$ is a polar space of rank $n'-m-1\ge n$
(in the case when $S$ is empty, this polar space 
coincides with $\Pi'$).
Every collinearity preserving injection of $\Pi$ to $\Pi'_{S}$
induces an isometric embedding of $\Gamma_{k}(\Pi)$ in $\Gamma_{k}(\Pi'_{S})$.
For every $k\in \{0,1,\dots,n-1\}$
this mapping  can be considered as an isometric embedding of 
$\Gamma_{k}(\Pi)$ in $\Gamma_{k'}(\Pi')$, where $k'=m+k+1$.

In this paper we will investigate  isometric embeddings of 
the polar Grassmann graph $\Gamma_{k}(\Pi)$ 
(this graph coincides with $\Gamma_{k}(n)$ if $\Pi$ is a thin polar space) in  
the polar Grassmann graph $\Gamma_{k'}(\Pi')$.
We start from the following simple observation.

\begin{prop}\label{prop0}
If $n\ge 4$ and $f$ is an isometric embedding of $\Gamma_{0}(\Pi)$ in $\Gamma_{m}(\Pi')$
then $m\le n'-n$ and there is an $(m-1)$-dimensional singular subspace $S$ in $\Pi'$
such that the image of $f$ is contained in $[S\rangle_{m}$ and 
$f$ is a collinearity preserving injection of $\Pi$ to $\Pi'_{S}$.
\end{prop}

\begin{rem}\label{rem0}{\rm
The diameter of $\Gamma_{0}(\Pi)$ is equal to $2$
and every embedding of this graph is isometric.
}\end{rem}

All isometric embeddings of 
the dual polar graph $\Gamma_{n-1}(\Pi)$ in the dual polar graph $\Gamma_{n'-1}(\Pi')$
are described in \cite{Pankov1}.
The existence of such embeddings implies that the diameter of $\Gamma_{n-1}(\Pi)$ 
is not greater than the diameter of $\Gamma_{n'-1}(\Pi')$, i.e. $n\le n'$.
By \cite[Theorem 2]{Pankov1}, the image of every isometric embedding of 
the $n$-dimensional hypercube graph $H_{n}=\Gamma_{n-1}(n)$ in 
the dual polar graph $\Gamma_{n'-1}(\Pi')$
is an apartment of ${\mathcal G}_{n-1}(\Pi'_{S})$,
where $S$ is an $(n'-n-1)$-dimensional singular subspace of $\Pi'$.
Using this result and \cite[Theorem 4.17]{Pankov-book1}
the author shows that 
every isometric embedding of $\Gamma_{n-1}(\Pi)$ in $\Gamma_{n'-1}(\Pi')$
is induced by a collinearity preserving injection of $\Pi$ to $\Pi'_{S}$,
as above, $S$ is an $(n'-n-1)$-dimensional singular subspace of $\Pi'$
\cite[Theorem 3]{Pankov1}.

We will consider the case when our polar Grassmann graphs both
are formed by non-maximal singular subspaces.
The first result concerns  the case when $n\ge 5$ and $1\le k \le n-4$.

\begin{theorem}\label{theorem-main1}
Suppose that $n\ge 5$.
If $f$ is an isometric embedding of $\Gamma_{k}(\Pi)$ in $\Gamma_{k'}(\Pi')$ and
$1\le k\le n-4$ then  
$$k\le k',\;\;n-k\le n'-k'$$
and there is a $(k'-k-1)$-dimensional singular subspace $S$ of $\Pi'$
such that the image of $f$ is contained in $[S\rangle_{k'}$
and $f$ is induced by a collinearity preserving injection of $\Pi$ to $\Pi'_{S}$.
\end{theorem}

It must be pointed out that in Theorem \ref{theorem-main1} 
there is no assumption  concerning $n'$ and $k'$.
The case when $n\ge 4$ and $k=n-3$ is different.

\begin{theorem}\label{theorem-main2}
Suppose that $n\ge 4$ and 
$f$ is an isometric embedding of $\Gamma_{n-3}(\Pi)$ in $\Gamma_{k'}(\Pi')$.
If $\Pi$ is a polar space of type {\rm$\textsf{C}_{n}$} then 
$$n-3\le k'\le n'-3$$
and there is a $(k'-n+2)$-dimensional singular subspace $S$ of $\Pi'$
such that the image of $f$ is contained in $[S\rangle_{k'}$
and $f$ is induced by a collinearity preserving injection of $\Pi$ to $\Pi'_{S}$.
In the case when $\Pi$  is a polar space of type {\rm$\textsf{D}_{n}$},
the following assertions are fulfilled:
\begin{enumerate}
\item[{\rm (1)}] If $n=4$ then $1\le k'\le n'-3$
and there is a $(k'-2)$-dimensional singular subspace $S$ of $\Pi'$
such that the image of $f$ is contained in $[S\rangle_{k'}$.
Also, there is an automorphism $g$ of $\Gamma_{1}(\Pi)$
$($possibility identity$)$ such that 
the composition $fg$ is induced by a collinearity preserving injection of $\Pi$ to $\Pi'_{S}$.
\item[{\rm (2)}]
If $n\ge 5$ and $k'=n'-3$ then 
$n\le n'$ and there is a $(n'-n-1)$-dimensional singular subspace $S$ of $\Pi'$
such that the image of $f$ is contained in $[S\rangle_{k'}$
and $f$ is induced by a collinearity preserving injection of $\Pi$ to $\Pi'_{S}$.
\end{enumerate}
\end{theorem}

Theorem \ref{theorem-main2} does not contain  any assumption concerning $n'$ and $k'$
except the case when $\Pi$ is a polar space of type $\textsf{D}_{n}$, $n\ge 5$.
In this special case, we can describe isometric embeddings 
of $\Gamma_{n-3}(\Pi)$ in $\Gamma_{n'-3}(\Pi')$ only.

Our third result covers the case when $n=n'$ and $k=k'=n-2$.

\begin{theorem}\label{theorem-main3}
If $n=n'$ then every isometric embedding of $\Gamma_{n-2}(\Pi)$ in $\Gamma_{n-2}(\Pi')$
is induced by a collinearity preserving injection of $\Pi$ to $\Pi'$.
\end{theorem}

As a direct consequence of the above results, 
we get the following characterization of apartments in polar Grassmannians.

\begin{cor}
The image of every isometric embedding of 
$\Gamma_{k}(n)$ in $\Gamma_{k}(\Pi)$ 
is an apartment of ${\mathcal G}_{k}(\Pi)$.
\end{cor}

\begin{proof}
For $k=0$ the statement follows directly from the frame definition. 
The case $k=n-1$ was considered in \cite[Theorem 2]{Pankov1}.
If $1\le k\le n-2$ then we apply Theorems \ref{theorem-main1}--\ref{theorem-main3} 
to isometric embeddings of $\Gamma_{k}(n)$ in $\Gamma_{k}(\Pi)$.
\end{proof}

\section{Proof of Proposition \ref{prop0} and Theorem \ref{theorem-main1}}

\subsection{Triangles}
We say that three distinct mutually adjacent vertices of $\Gamma_{k}(\Pi)$
form a {\it triangle} if they do not belong to a common line of ${\mathcal G}_{k}(\Pi)$.
The existence of triangles implies that $k\le n-2$.
If $1\le k\le n-3$ then there are the following two types of triangles:
{\it star-triangles} contained in stars and {\it top-triangles} contained in tops
\cite[Lemma 4.10]{Pankov-book1}.
Note that ${\mathcal G}_{n-2}(\Pi)$ contains only top-triangles.
If $S_{1},S_{2},S_{3}\in {\mathcal G}_{k}(\Pi)$ form a star-triangle then
$$\dim (S_{1}\cap S_{2}\cap S_{3})=k-1\;\mbox{ and }\;
\dim\langle S_{1},S_{2},S_{3}\rangle=k+2.$$
In the case when $S_{1},S_{2},S_{3}\in {\mathcal G}_{k}(\Pi)$
form a top-triangle, we have
$$\dim (S_{1}\cap S_{2}\cap S_{3})=k-2\;\mbox{ and }\;
\dim\langle S_{1},S_{2},S_{3}\rangle=k+1.$$
Each triangle of $\Pi$ is a star-triangle.

\begin{lemma}\label{lemma-triangle}
If $n\ge 3$ and $k\le n-3$ then every embedding of $\Gamma_{k}(\Pi)$ in $\Gamma_{k'}(\Pi')$ 
transfers triangles to triangles.
\end{lemma}

\begin{proof}
Let $f$ be an embedding of $\Gamma_{k}(\Pi)$ in $\Gamma_{k'}(\Pi')$ 
and $k\le n-3$. 
Suppose that $S_{1},S_{2},S_{3}\in {\mathcal G}_{k}(\Pi)$ form a triangle
and ${\mathcal X}$ is a maximal clique of $\Gamma_{k}(\Pi)$
containing this triangle.
There exists a maximal clique ${\mathcal Y}$ of $\Gamma_{k}(\Pi)$
intersecting ${\mathcal X}$ precisely in  the line joining $S_{1}$ and $S_{2}$.
If $1\le k\le n-3$ then maximal cliques of $\Gamma_{k}(\Pi)$ are stars and tops and
the latter statement is obvious.
Maximal cliques of $\Gamma_{0}(\Pi)$ are maximal singular subspaces of $\Pi$
and it is well-known that for any singular subspace $S$
there exist maximal singular subspaces $M$ and $M'$
such that $S=M\cap M'$.

So, $S_{3}\not\in {\mathcal Y}$ and 
${\mathcal Y}$ contains a vertex $Y$ non-adjacent to $S_{3}$.
Since $f(Y)$ is adjacent to $f(S_{1})$ and $f(S_{2})$,
it is adjacent to all vertices of $\Gamma_{k'}(\Pi')$ belonging to the line
joining $f(S_{1})$ and $f(S_{2})$.
On the other hand, $f(Y)$ and $f(S_{3})$ are not adjacent.
Therefore, $f(S_{3})$ is not on this line.
\end{proof}

\begin{lemma}
If $n\ge 3$ and $k\le n-3$ then
for any embedding of $\Gamma_{k}(\Pi)$ in $\Gamma_{k'}(\Pi')$ 
the image of every maximal clique of $\Gamma_{k}(\Pi)$
cannot be contained in two maximal cliques of different types.
\end{lemma}

\begin{proof}
In this case, every maximal clique of $\Gamma_{k}(\Pi)$ contains a triangle.
On the other hand, the intersection of two maximal cliques of different types
is empty or a one-element set or a line. 
Lemma \ref{lemma-triangle} gives the claim.
\end{proof}

\subsection{Proof of Proposition \ref{prop0}}
Suppose that $f$ is an embedding of $\Gamma_{0}(\Pi)$ in $\Gamma_{m}(\Pi')$
(see Remark \ref{rem0}) and $n\ge 4$. 

Let $M$ be a maximal singular subspace of $\Pi$.
We take any maximal singular subspace $M'$ such that 
$M\cap M'$ is $(n-2)$-dimensional.
Consider two maximal cliques of $\Gamma_{m}(\Pi')$ containing $f(M)$ and $f(M')$.
By Subsection 2.1, they are distinct.
The intersection of these cliques contains more than one element.
Moreover, this intersection is a line of ${\mathcal G}_{m}(\Pi')$ 
if one of the cliques is a top.
Then $f(M\cap M')$ is contained in a line. 
The latter contradicts Lemma \ref{lemma-triangle},
since the dimension of $M\cap M'$ is equal to $n-2\ge 2$.

Thus $f(M)$ is contained in a star. 
Hence there exists $S\in {\mathcal G}_{m-1}(\Pi')$
such that
$$f(M)\subset [S\rangle_{m}.$$
Let $p\in P\setminus M$.
It follows from the axiom (P3) that 
the set of all points of $M$ collinear to $p$ is an $(n-2)$-dimensional singular subspace.
Denote by $M'$ the maximal singular subspace spanned by this subspace and the point $p$.
As above, we obtain that 
$$f(M')\subset [S'\rangle_{m}$$
for a certain $S'\in {\mathcal G}_{m-1}(\Pi')$.
The singular subspace $M\cap M'$ contains more than one point
and for any distinct points $t,q\in M\cap M'$ we have
$$S=f(t)\cap f(q)=S'.$$
Therefore, $f(p)\in [S\rangle_{m}$ for every point $p\in P$.

So, the image of $f$ is contained in $[S\rangle_{m}$
and $f$ is a collinearity preserving injection of $\Pi$ to $\Pi'_{S}$.
The rank of $\Pi'_{S}$ is equal to $n'-m$ and it is not less than $n$ which implies 
that $m\le n'-n$.

\subsection{Technical result}
Let $f$ be an isometric embedding of $\Gamma_{k}(\Pi)$ in $\Gamma_{k'}(\Pi')$ 
and $1\le k\le n-3$.
Then maximal cliques of $\Gamma_{k}(\Pi)$ are stars and tops
and there exist pairs of distinct maximal cliques whose intersections contain more than 
one element.
In the case when $k'\ge n'-2$,
there is only one type of maximal cliques in $\Gamma_{k'}(\Pi')$ and
the intersection of any two distinct maximal cliques contains at most one element.
It was noted in Subsection 2.1 that 
$f$ sends distinct maximal cliques to subsets of distinct maximal cliques.
This guarantees that $k'\le n'-3$.
Also, the existence of isometric embeddings of $\Gamma_{k}(\Pi)$ in $\Gamma_{k'}(\Pi')$
implies that the diameter of $\Gamma_{k}(\Pi)$ is not greater than 
the diameter of $\Gamma_{k'}(\Pi')$.
By Subsection 2.4. the diameters of these graphs are equal to $k+2$ and $k'+2$, respectively. 
Therefore, $k\le k'$.

\begin{prop}\label{prop-tech}
If $f$ transfers stars to subsets of stars
then 
$$n-k\le n'-k'$$
and there exists a $(k'-k-1)$-dimensional singular subspace $S$ of $\Pi'$
such that the image of $f$ is contained in $[S\rangle_{k'}$
and $f$ is induced by a collinearity preserving injection of $\Pi$ to $\Pi'_{S}$.
\end{prop}

Proposition \ref{prop-tech} will be proved in several steps.
Our first step is the following.

\begin{lemma}\label{lemma1-1}
$f$ transfers big stars to subsets of big stars.
\end{lemma}

\begin{proof}
Let $S\in {\mathcal G}_{k-1}(\Pi)$.
The big star $[S\rangle_{k}$ is the union of all stars $[S,U]_{k}$,
where $U$ is a maximal singular subspace of $\Pi$ containing $S$.
By our assumption, $f$ transfers stars $[S,X]_{k}$ and $[S,Y]_{k}$
to subsets contained in some stars $[S',X']_{k'}$ and $[S'',Y']_{k'}$, respectively.
If $X$ and $Y$ are adjacent vertices of $\Gamma_{n-1}(\Pi)$, 
i.e. the dimension of $X\cap Y$ is equal to $n-2\ge k+1$, 
then the intersection of $[S,X]_{k}$ and $[S,Y]_{k}$ contains more than one element
and the same holds for the intersection of $[S',X']_{k'}$ and $[S'',Y']_{k'}$.
The latter is possible only in the case when $S'=S''$.
Since the restriction of the graph $\Gamma_{n-1}(\Pi)$ to $[S\rangle_{n-1}$ is connected, 
there exists $S'\in {\mathcal G}_{k'-1}(\Pi')$ such that 
the images of all stars $[S,X]_{k}$ are contained in some stars of type $[S',X']_{k'}$.
This means that $f$ transfers the big star $[S\rangle_{k}$ 
to a subset of the big star $[S'\rangle_{k'}$.
\end{proof}

The intersection of two distinct big stars contains at most one element.
This implies that for any $S\in {\mathcal G}_{k-1}(\Pi)$
there is unique $S'\in{\mathcal G}_{k'-1}(\Pi')$ such that 
$f([S\rangle_{k})$ is contained in $[S'\rangle_{k'}$.
So, our embedding induces a mapping 
$$f_{k-1}:{\mathcal G}_{k-1}(\Pi)\to {\mathcal G}_{k'-1}(\Pi')$$
such that 
$$f([S\rangle_{k})\subset[f_{k-1}(S)\rangle_{k'}$$
for every $S\in {\mathcal G}_{k-1}(\Pi)$.

\begin{lemma}\label{lemma1-2}
The mapping $f_{k-1}$ is injective. 
\end{lemma}

\begin{proof}
Let $S$ and $U$ be distinct elements of ${\mathcal G}_{k-1}(\Pi)$.
We take any frame of $\Pi$ such that $S$ and $U$ are spanned by subsets of this frame.
The associated apartment ${\mathcal A}\subset{\mathcal G}_{k}(\Pi)$ 
contains $X\in [S\rangle_{k}$ and $Y\in [U\rangle_{k}$ satisfying  $d(X,Y)\ge 3$.
Indeed, if the dimension of $S\cap U$ is less than $k-2$ 
then we choose any $X\in {\mathcal A}\cap[S\rangle_{k}$ and 
$Y\in {\mathcal A}\cap[U\rangle_{k}$
such that 
$$X\cap Y=S\cap U.$$
In the case when $S\cap U$ is $(k-2)$-dimensional,
we require in addition that every point of 
$X\setminus Y$ is non-collinear to a certain point of $Y$.
See Lemma \ref{lemma-dist}.

If $f_{k-1}(S)$ coincides with $f_{k-1}(U)$ then 
the intersection of $f(X)$ and $f(Y)$ is $(k'-1)$-dimensional.
By Lemma \ref{lemma-dist},
$$d(f(X),f(Y))\le 2$$
which contradicts the fact that $f$ is an isometric embedding 
of $\Gamma_{k}(\Pi)$ in $\Gamma_{k'}(\Pi')$. 
\end{proof}

For every $U\in {\mathcal G}_{k}(\Pi)$ we have
$$f_{k-1}(\langle U]_{k-1})\subset \langle f(U)]_{k'-1},$$
i.e. $f_{k-1}$ transfers tops to subsets of tops
which implies that $f_{k-1}$ sends adjacent vertices of $\Gamma_{k-1}(\Pi)$
to adjacent vertices of $\Gamma_{k'-1}(\Pi')$. 
However, we cannot state that non-adjacent vertices of $\Gamma_{k-1}(\Pi)$
go to non-adjacent vertices of $\Gamma_{k'-1}(\Pi')$.

Suppose that $k\ge 2$. 
Then $f_{k-1}$ transfers maximal cliques of $\Gamma_{k-1}(\Pi)$
(stars and tops) to subsets of maximal cliques of $\Gamma_{k'-1}(\Pi')$. 
We do not show that $f_{k-1}$ is an isometric embedding of $\Gamma_{k-1}(\Pi)$
in $\Gamma_{k'-1}(\Pi')$. 
It is sufficient to prove the following.

\begin{lemma}\label{lemma1-3}
$f_{k-1}$ transfers stars to subsets of stars.
\end{lemma}

\begin{proof}
Suppose that there exists a star ${\mathcal S}\subset {\mathcal G}_{k-1}(\Pi)$
whose image is not contained in a star of ${\mathcal G}_{k'-1}(\Pi')$.
Then $f_{k-1}({\mathcal S})$ is a subset in a certain top 
${\mathcal T}\subset {\mathcal G}_{k'-1}(\Pi')$.
We choose distinct $U_{1},U_{2}\in {\mathcal G}_{k}(\Pi)$
such that for every $i=1,2$ the top $\langle U_{i}]_{k-1}$
intersects the star ${\mathcal S}$ in a line.
Then the intersection of $\langle f(U_{i})]_{k'-1}$ and ${\mathcal T}$ contains more than one element.
This is possible only in the case when 
the tops $\langle f(U_{i})]_{k'-1}$ and ${\mathcal T}$ are coincident.
Hence $f(U_{1})=f(U_{2})$ which contradicts the fact that $f$ is injective. 
\end{proof}

Using Lemmas \ref{lemma1-2}, \ref{lemma1-3}
and the arguments from the proof of Lemma \ref{lemma1-1}, we show that $f_{k-1}$ 
transfers big stars to subsets of big stars
and the image of every big star of ${\mathcal G}_{k-1}(\Pi)$
is contained in a unique big star of ${\mathcal G}_{k'-1}(\Pi')$.
So, $f_{k-1}$ induces a mapping 
$$f_{k-2}:{\mathcal G}_{k-2}(\Pi)\to{\mathcal G}_{k'-2}(\Pi').$$
If $S$ and $U$ are distinct elements of ${\mathcal G}_{k-2}(\Pi)$ 
then there exist $X\in [S\rangle_{k}$ and $Y\in [U\rangle_{k}$ satisfying  
$d(X,Y)\ge 4$ (as in the proof of Lemma \ref{lemma1-2}, 
we take a frame of $\Pi$ such that $S$ and $U$ are spanned by subsets of this frame
and choose $X,Y$ in the associated apartment of ${\mathcal G}_{k}(\Pi)$).
It is clear that 
$$f_{k-2}(S)\subset f(X)\;\mbox{ and }\;f_{k-2}(U)\subset f(Y).$$
If $f_{k-2}(S)$ coincides with $f_{k-2}(U)$ then 
the dimension of the intersection of $f(X)$ and $f(Y)$ is not less than $k'-2$.
Lemma \ref{lemma-dist} shows that
$$d(f(X),f(Y))\le 3$$
and $f$ is not an isometric embedding of $\Gamma_{k}(\Pi)$ in $\Gamma_{k'}(\Pi')$.
Therefore, $f_{k-2}$ is injective.

The mapping $f_{k-2}$ transfers tops to subsets of tops.
As in the proof of Lemma \ref{lemma1-3},
we show that $f_{k-2}$ sends stars to subsets of stars if $k\ge 3$.
Step by step, we obtain a sequence of injections 
$$f_{i}:{\mathcal G}_{i}(\Pi)\to {\mathcal G}_{k'-k+i}(\Pi'),\;\;\;i=k,k-1,\dots,0$$
such that $f_{k}=f$ and 
$$f_{i}([S\rangle_{i})\subset [f_{i-1}(S)\rangle_{k'-k+i}$$
for every $S\in {\mathcal G}_{i-1}(\Pi)$ if $i\ge 1$.
Then 
$$
f_{i-1}(\langle U]_{i-1})\subset \langle f_{i}(U)]_{k'-k+i-1}$$
for every $U\in {\mathcal G}_{i}(\Pi)$.
Also, we have
$$
f_{0}(\langle U]_{0})\subset \langle f_{i}(U)]_{k'-k}
$$
for all $U\in {\mathcal G}_{i}(\Pi)$ and $i\le k$.
Indeed, if $U\in {\mathcal G}_{i}(\Pi)$ then $\langle U]_{0}$ 
is the union of all $\langle S]_{0}$ such that $S\in \langle U]_{i-1}$ and
the latter inclusion can be proved by induction.

\begin{lemma}\label{lemma1-4}
There exists $S\in {\mathcal G}_{k'-k-1}(\Pi')$
such that the image of $f$ is contained in the big star $[S\rangle_{k'}$.
\end{lemma}

\begin{proof}
The case when $k'=k$ is trivial and we suppose that $k'>k$. 
The mapping $f_{0}$ is an injection sending lines of $\Pi$
to subsets in tops of ${\mathcal G}_{k'-k}(\Pi')$.
Hence it transfers collinear points of $\Pi$
to adjacent vertices of $\Gamma_{k'-k}(\Pi')$.
Let $M$ be a maximal singular subspace of $\Pi$.
Then $f_{0}(M)$ is a clique of $\Gamma_{k'-k}(\Pi')$.
If $f_{0}(M)$ is contained in a top
then $f_{0}$ transfers all lines of $M$ to subsets of the same top
which contradicts the fact that $f_{1}$ is injective.
Thus $f_{0}(M)$ is a subset in a certain star and 
there exists $S\in {\mathcal G}_{k'-k-1}(\Pi')$ such that
$$f_{0}(M)\subset [S\rangle_{k'-k}.$$
As in the proof of Proposition \ref{prop0} (Subsection 4.2),
we show that $f_{0}(p)$ belongs to $[S\rangle_{k'-k}$ for every point $p\in P$.
If $U\in {\mathcal G}_{k}(\Pi)$ and $p$ is a point of $U$ then 
$$S\subset f_{0}(p)\subset f(U),$$
i.e. $f(U)$ belongs to $[S\rangle_{k'}$.
\end{proof}

Therefore, $f$ is an isometric embedding of $\Gamma_{k}(\Pi)$
in $\Gamma_{k}(\Pi'_{S})$ and $f_{0}$ is an injection of $P$ to $[S\rangle_{k'-k}$ 
transferring lines of $\Pi$ to subsets in lines of $\Pi'_{S}$.

\begin{lemma}\label{lemma1-5}
Let $U\in {\mathcal G}_{i}(\Pi)$ and $i\le k$.
If $U$ is spanned by points $p_{1},\dots,p_{i+1}$
then $f_{i}(U)$ is spanned by 
$f_{0}(p_{1}),\dots,f_{0}(p_{i+1})$.
\end{lemma}

\begin{proof}
We prove the statement by induction. The case when $i=0$ is trivial.
Suppose that $i\ge 1$ and 
consider the $(i-1)$-dimensional singular subspaces $M$ and $N$
spanned by $p_{1},\dots,p_{i}$ and $p_{2},\dots,p_{i+1}$,
respectively. 
By the inductive hypothesis, $f_{i-1}(M)$ and $f_{i-1}(N)$
are spanned by  
$$f_{0}(p_{1}),\dots,f_{0}(p_{i})\;\mbox{ and }\;f_{0}(p_{2}),\dots,f_{0}(p_{i+1}),$$
respectively.
The required statement follows from the fact that $f_{i-1}(M)$ and $f_{i-1}(N)$ are 
distinct $(k'-k+i-1)$-dimensional singular subspaces 
contained in the $(k'-k+i)$-dimensional singular subspace $f_{i}(U)$.
\end{proof}

Our last step is to show that $f_{0}$ sends non-collinear points of $\Pi$ to 
non-collinear points of $\Pi'_{S}$.

Let $p$ and $q$ be non-collinear points of $\Pi$.
Consider a frame 
$$p_{1}=p,p_{2},\dots,p_{n},q_{1}=q,q_{2},\dots,q_{n},$$
where every $p_{i}$ is non-collinear to $q_{i}$.
Denote by $X$ and $Y$  the $k$-dimensional singular subspaces spanned by 
$p_{1},\dots,p_{k+1}$ and $q_{1},\dots,q_{k+1}$, respectively.
Then 
$$d(X,Y)=k+2$$
(Lemma \ref{lemma-dist}).
By Lemma \ref{lemma1-5}, $f(X)$ and $f(Y)$ are $k$-dimensional subspaces of $\Pi'_{S}$
spanned by 
$$f_{0}(p_{1}),\dots, f_{0}(p_{k+1})\;\mbox{ and }\;f_{0}(q_{1}),\dots,f_{0}(q_{k+1}),$$
respectively.
The point $p_{1}$ is collinear to $q_{i}$ if $i\ne 1$.
Hence $f_{0}(p_{1})$ and $f_{0}(q_{i})$ are collinear points of $\Pi'_{S}$
if $i\ne 1$.
Therefore, if $f_{0}(p_{1})$ is collinear to $f_{0}(q_{1})$ 
then it is collinear to all points of $f(Y)$ which means that 
$$d(f(X),f(Y))\le k+1$$
(Lemma \ref{lemma-dist}).
The latter is impossible, since $f$ is an isometric embedding of 
$\Gamma_{k}(\Pi)$ in $\Gamma_{k'}(\Pi')$.
Thus $f_{0}(p)$ and $f_{0}(q)$ are non-collinear points of $\Pi'_{S}$.

So, $f_{0}$ is a collinearity preserving injection of $\Pi$ to $\Pi'_{S}$.
It follows from Lemma \ref{lemma1-5} that $f(U)$ coincides with 
$\langle f_{0}(U)\rangle$ for every $U\in {\mathcal G}_{k}(\Pi)$, 
i.e. $f$ is induced by $f_{0}$.
The rank of $\Pi'_{S}$ is equal to $n'-k'+k$.
The existence of collinearity preserving injections of $\Pi$ to $\Pi'_{S}$
implies that $n'-k'+k\ge n$.

\subsection{Proof of Theorem \ref{theorem-main1}}
Let $f$ be as in the previous subsection.
We need to show that $f$ transfers stars to subsets of stars
if $n\ge 5$ and $k\le n-4$. 

Suppose that there is a star ${\mathcal S}\subset{\mathcal G}_{k}(\Pi)$
such that $f({\mathcal S})$ is contained in a top of  ${\mathcal G}_{k'}(\Pi')$.
If $X,Y,Z\in {\mathcal S}$ form a triangle then 
their images form a top-triangle. 
The corresponding top is the unique maximal clique of $\Gamma_{k'}(\Pi')$
containing $f(X),f(Y),f(Z)$.
On the other hand, 
the singular subspace spanned by $X,Y,Z$
is $(k+2)$-dimensional.
Since $k\le n-4$, this singular subspace is not maximal.
This guarantees the existence 
of a star ${\mathcal S}'$ containing $X,Y,Z$ and different from ${\mathcal S}$.
By Subsection 2.1, $f({\mathcal S})$ and $f({\mathcal S}')$ are subsets of distinct 
maximal cliques of $\Gamma_{k'}(\Pi')$. 
Each of these cliques contains $f(X),f(Y),f(Z)$ and we get a contradiction.

\section{Proof of Theorem \ref{theorem-main2}}
In this section we suppose that $f$ is an isometric embedding of 
$\Gamma_{n-3}(\Pi)$ in $\Gamma_{k'}(\Pi')$
and $n\ge 4$. 
By Subsection 4.3, we have $n-3\le k'\le n'-3$.

\subsection{Regular pairs of triangles}
Let
$$\Delta=\{S_{1},S_{2},S_{3}\}\;\mbox{ and }\;
\Delta'=\{S'_{1},S'_{2},S'_{3}\}$$
be triangles in ${\mathcal G}_{k}(\Pi)$ such that $\Delta\cap\Delta'=\emptyset$.
We say that these triangles form a {\it regular pair} if 
$S_{i}$ and $S'_{j}$ are adjacent vertices of $\Gamma_{k}(\Pi)$ 
only in the case when $i\ne j$, in other words, 
every vertex from each of these triangles is adjacent to precisely two vertices of 
the other triangle. 
An easy verification shows that  
in this case one of the following possibilities is realized:
\begin{enumerate}
\item[(1)] There are $(k+2)$-dimensional singular subspaces $U$ and $U'$
whose intersection $S$ is $(k-1)$-dimensional 
and $\Delta,\Delta'$ are star-triangles contained in 
$[S,U]_{k}$ and $[S,U']_{k}$, respectively.
Note that for every point $q\in U\setminus S$ there is a point of $U'$
non-collinear to $q$. 
Similarly, for every point $q'\in U'\setminus S$ there is a point of $U$
non-collinear to $q'$.
\item[(2)] One of the triangles is a star-triangle and the other is a top-triangle.
For example, if $\Delta$ is a star-triangle and $\Delta'$ is a top-triangle then
the singular subspace $\langle S'_{1},S'_{2},S'_{3}\rangle$ is 
$(k+1)$-dimensional and 
there is a point $p\not\in\langle S'_{1},S'_{2},S'_{3}\rangle$
collinear to all points of $\langle S'_{1},S'_{2},S'_{3}\rangle$
and such that 
$$S_{1}=\langle p, S'_{2}\cap S'_{3}\rangle,\;
S_{2}=\langle p, S'_{1}\cap S'_{3}\rangle,\; S_{3}=\langle p, S'_{1}\cap S'_{2}\rangle.$$
Note that all elements of our triangles are contained in 
the $(k+2)$-dimensional singular subspace spanned by $p$
and  $\langle S'_{1},S'_{2},S'_{3}\rangle$.
\end{enumerate}
It is clear that $f$ transfers regular pairs of triangles to regular pairs of triangles.

\subsection{Proof of Theorem \ref{theorem-main2} for $n=4$}
Suppose that $n=4$.
Then $f$ is an isometric embedding of $\Gamma_{1}(\Pi)$ in $\Gamma_{k'}(\Pi')$.
A maximal singular subspace $U\in {\mathcal G}_{3}(\Pi)$
is said to be {\it special} if there exists a point $p\in U$ such that
$f$ transfers the star $[p,U]_{1}$ to a subset contained in a top.
In the case when there exist no special maximal singular subspaces,
we apply Proposition \ref{prop-tech}.

\begin{lemma}\label{lemma2-2}
If $U\in {\mathcal G}_{3}(\Pi)$ is special then 
for every point $q\in U$ the image of $[q,U]_{1}$ is contained in a top
and for every $2$-dimensional singular subspace $S\subset U$
the image of $\langle S]_{1}$ is a subset in a star.
\end{lemma}

This statement is proved for the case when $k'=1$ and $n'=4$ 
\cite[Lemma 4.15]{Pankov-book1}.
Now we show that the same arguments work in the general case.

\begin{proof}
Let $p$ be a point of $U$ such that
$f([p,U]_{1})$ is contained in a top
$$\langle S']_{k'},\;\;S'\in {\mathcal G}_{k'+1}(\Pi').$$
We take any $2$-dimensional singular subspace $S\subset U$
which does not contain the point $p$. 
Consider a regular pair of triangles
$$\Delta\subset [p,U]_{1}\;\mbox{ and }\;\Delta'\subset \langle S]_{1}.$$
Since $f(\Delta)$ is a top-triangle, 
the triangles $f(\Delta)$ and $f(\Delta')$ form a regular pair of type (2).
Therefore, $f(\Delta')$ is a star-triangle and 
$f(\langle S]_{1})$ is contained in a certain star 
\begin{equation}\label{eq2-1}
[T',U']_{k'},\;\;T'\in {\mathcal G}_{k'-1}(\Pi'),\;U'\in {\mathcal G}_{n'-1}(\Pi').
\end{equation}
Note that $S'\subset U'$ and $T'\not\subset S'$.

Let $q$ be a point belonging to $U\setminus\{p\}$. 
We choose a $2$-dimensional singular subspace $S\subset U$
which does not contain $p$ and $q$.
It was established above that $f(\langle S]_{1})$ is a subset in a certain star \eqref{eq2-1}.
Consider a regular pair of triangles
$$\Delta\subset [q,U]_{1},\;\mbox{ and }\;\Delta'\subset \langle S]_{1}.$$
Then $f(\Delta')$ is a star-triangle. 
Suppose that the triangles $f(\Delta)$ and $f(\Delta')$ form a regular pair of type (1).
Then $f([q,U]_{1})$ is contained in $[T', U'']_{k'}$
for a certain $U''\in {\mathcal G}_{n'-1}(\Pi')$.
Since $T'\not\subset S'$, 
the top $\langle S']_{k'}$ and the star $[T', U'']_{k'}$ are disjoint.
On the other hand, 
$$f([p,U]_{1})\subset \langle S']_{k'},\;\;f([q,U]_{1})\subset[T', U'']_{k'}$$
and the stars $[p,U]_{1}$, $[q,U]_{1}$ both contain the line joining $p$ and $q$.
This contradiction shows that $f(\Delta)$ and $f(\Delta')$ form a regular pair of type (2),
i.e. $f(\Delta)$ is a top-triangle. 
Then the image of $[q,U]_{1}$ is contained in a top.

Let $S$ be a $2$-dimensional singular subspace of $U$ containing the point $p$.
We take any point $q\in U\setminus S$. 
Then $f([q,U]_{1})$ is contained in a top. 
As above, we establish that $f(\langle S]_{1})$ is a subset in a star.
\end{proof}

\begin{lemma}\label{lemma2-3}
Let $U,Q\in \mathcal{G}_{3}(\Pi)$. If $U$ is special and $\dim(U\cap Q)=2$
then $Q$ is not special.
\end{lemma}

\begin{proof}
We take two distinct lines $L_{1},L_{2}\subset U\cap Q$ and 
consider star-triangles 
$$\Delta=\{L_{1},L_{2},L_{3}\}\;\mbox{ and }\;
\Delta'=\{L_{1},L_{2},L'_{3}\}$$
contained in $\langle U]_{1}$ and $\langle Q]_{1}$, respectively.
The lines $L_{3}$ and $L'_{3}$ have a common point;
on the other hand, they contain non-collinear points
which means that the distance between them in $\Gamma_{1}(\Pi)$ is equal to $2$.
Thus 
$$d(f(L_{3}),f(L'_{3}))=2.$$
By our assumption, $f(\Delta)$ is a top-triangle.
If the same holds for $f(\Delta')$ then $f(L'_{3})$ belongs to the top 
containing $f(\Delta)$ which implies that $f(L_{3})$ and $f(L'_{3})$
are adjacent vertices of $\Gamma_{k'}(\Pi')$.
So, $f(\Delta')$ is a star-triangle. 
Then $f$ transfer the star containing $\Delta'$ to a subset in a star.
By Lemma \ref{lemma2-2}, $Q$ is not special.
\end{proof}

\begin{lemma}\label{lemma2-4} 
Let $U,Q\in \mathcal{G}_{3}(\Pi)$. If $U$ is special and $\dim(U\cap Q)=1$ then 
$Q$ is special.
\end{lemma}

\begin{proof}
The intersection of $U$ and $Q$ is a line. We denote this line by $L$.
Consider the star-triangles 
$$\Delta=\{L,L_{1},L_{2}\}\subset \langle U]_{1}\;\mbox{ and }\;
\Delta'=\{L,L'_{1},L'_{2}\}\subset \langle Q]_{1}$$
such that $L_{i}$ and $L'_{j}$ are adjacent vertices of $\Gamma_{1}(\Pi)$
only in the case when $i=j$.
Then $f(\Delta)$ is a top-triangle.
We observe that 
$$f(L),f(L_{1}),f(L'_{1})$$
form a triangle. 
This is not a top-triangle (otherwise, $f(L'_{1})$ belongs to 
the top containing $f(\Delta)$ which is impossible).
Therefore, $f(L'_{1})$ contains the $(k'-1)$-dimen\-sional singular subspace 
$f(L)\cap f(L_{1})$.
Similarly, we establish that
$$f(L),f(L_{2}),f(L'_{2})$$
form a star-triangle
and $f(L'_{2})$ contains the $(k'-1)$-dimensional singular subspace 
$f(L)\cap f(L_{2})$. 
Since $f(\Delta)$ is a top-triangle, we have 
$$f(L)\cap f(L_{1})\ne f(L)\cap f(L_{2}).$$
So, $f(L'_{1})$ and $f(L'_{2})$ intersect $f(L)$ in two 
distinct $(k'-1)$-dimensional singular subspaces
which means that $f(\Delta')$ is a top-triangle.
Then $f$ sends the star containing $\Delta'$ to a subset in a top, 
i.e. $Q$ is special.
\end{proof}

Let $X$ and $Y$ be adjacent vertices of $\Gamma_{3}(\Pi)$.
Suppose that $\Pi$ is a polar space of type $\textsf{C}_{4}$.
Then the line joining $X$ and $Y$ contains a certain $Z\in {\mathcal G}_{3}(\Pi)$
distinct from $X,Y$.
We take any $T\in {\mathcal G}_{3}(\Pi)$ which does not belong to this line 
and such that $Z,T$ are adjacent vertices of $\Gamma_{3}(\Pi)$.
Then 
$$\dim (X\cap T)=\dim(Y\cap T)=1.$$
If $X$ is special then $T$ is special by Lemma \ref{lemma2-4}.
We apply Lemma \ref{lemma2-4} to $T,Y$ and establish that $Y$
is spacial. The latter is impossible by Lemma \ref{lemma2-3}.

Therefore, the existence of special maximal singular subspaces 
implies that $\Pi$ is a polar space of type $\textsf{D}_{4}$. 
Also, it follows from Lemmas \ref{lemma2-3} and \ref{lemma2-4}
that all special maximal singular subspaces 
form one of the half-spin Grassmannians ${\mathcal G}_{\delta}(\Pi)$, $\delta\in \{+,-\}$. 
Every $2$-dimensional singular subspace of $\Pi$ is contained 
in a certain element of ${\mathcal G}_{\delta}(\Pi)$.
Thus $f$ transfers every top to a subset in a star.
Let $g$ be the automorphism of $\Gamma_{1}(\Pi)$ described in Subsection 2.5.
Then the composition $fg$ is an isometric embedding of $\Gamma_{1}(\Pi)$ in $\Gamma_{k'}(\Pi')$
transferring stars to subsets of stars.

\subsection{Proof of Theorem \ref{theorem-main2} for $n\ge 5$}
Let $n\ge 5$. 
Suppose that there is a star $[S,U]_{n-3}$ in ${\mathcal G}_{n-3}(\Pi)$ 
such that $f([S,U]_{n-3})$ is contained in a top.
The singular subspace $S$ is $(n-4)$-dimensional and we take any $(n-5)$-dimensional singular 
subspace $T\subset S$.  
The rank of the polar space $\Pi_{T}$ is equal to $4$ and 
${\mathcal G}_{1}(\Pi_{T})$ coincides with $[T\rangle_{n-3}$.
The star $[S,U]_{n-3}$ is contained in $[T\rangle_{n-3}$, i.e.
it is a star of  ${\mathcal G}_{1}(\Pi_{T})$.
The restriction of $f$ to $[T\rangle_{n-3}$
is an isometric embedding of $\Gamma_{1}(\Pi_{T})$ in $\Gamma_{k'}(\Pi')$.
By Subsection 5.2, $\Pi_{T}$ is a polar space of type $\textsf{D}_{4}$
which implies that $\Pi$  is a polar space of type $\textsf{D}_{n}$.
Thus $f$ transfers stars to subsets of stars if $\Pi$  is a polar space of type $\textsf{C}_{n}$.

Now we consider the case when $k'=n'-3$ 
and show that $f$ sends every star to a subset in a star.

Suppose that ${\mathcal S}\subset {\mathcal G}_{n-3}(\Pi)$ is a star
such that $f({\mathcal S})$ is contained in a top.
We take any top 
$$\langle U]_{n-3},\;\;\;U\in {\mathcal G}_{n-2}(\Pi)$$
intersecting ${\mathcal S}$ in a line.
Since the intersection of two distinct tops contains at most one element,
$f(\langle U]_{n-3})$ cannot be in a top.
Hence it is contained in a certain star 
$$[S',U']_{n'-3},\;\;\;S'\in {\mathcal G}_{n'-4}(\Pi'),\;U'\in {\mathcal G}_{n'-1}(\Pi').$$
If $X_{1},X_{2},X_{3}\in \langle U]_{n-3}$ form a triangle 
then their images form a star-triangle and 
$$
\langle f(X_{1}), f(X_{2}), f(X_{3})\rangle =U'.
$$
The dimension of 
$$S:=X_{1}\cap X_{2}\cap X_{3}$$ is equal to $n-5$.
We choose $Y\in {\mathcal G}_{n-3}(\Pi)$ satisfying the following conditions:
\begin{enumerate}
\item[$\bullet$] $Y\cap U=S$,
\item[$\bullet$] there is a point of $Y\setminus U$ collinear to all points of $U$
(note that the singular subspace $U$ is not maximal).
\end{enumerate}
By Lemma \ref{lemma-dist}, $d(Y,X_{i})=2$ for every $i$ and we have
$d(Y,Z)=3$ for every $Z\in \langle U]_{n-3}$
which does not contain $S$.
We want to show that 
\begin{equation}\label{eq2-2}
d(f(Y), X')\le 2\;\;\;\;\;\forall\;X'\in [S',U']_{n'-3}.
\end{equation}
This contradicts the fact that $f$ is an isometric embedding of $\Gamma_{n-3}(\Pi)$ in 
$\Gamma_{n'-3}(\Pi')$ and we get the claim.

For every $i\in \{1,2,3\}$ we have
$$d(f(X_{i}),f(Y))=2.$$
It follows from Lemma \ref{lemma-dist} that 
the dimension of every $f(X_{i})\cap f(Y)$ is equal to $n'-4$ or $n'-5$.

First, we consider the case when $f(X_{i})\cap f(Y)$ is $(n'-4)$-dimensional 
for a certain $i$.
If $f(X_{i})\cap f(Y)$ coincides with $S'$
then $f(Y)$ belongs to the big star $[S'\rangle_{n'-3}$ which implies \eqref{eq2-2}.
If $f(X_{i})\cap f(Y)$ is distinct from $S'$
then 
$$\dim (f(Y)\cap S')=n'-5$$
and there is a point $$p\in (f(X_{i})\cap f(Y))\setminus S'.$$
This point belongs to $U'$.
Hence $p\in f(Y)\setminus S'$ is collinear to all points of every element of $[S',U']_{n'-3}$.
Since $f(Y)\cap S'$ is $(n'-5)$-dimensional,
we get  \eqref{eq2-2} again (Lemma \ref{lemma-dist}).

Now we suppose that $f(X_{i})\cap f(Y)$ is $(n'-5)$-dimensional for every $i$.
If this subspace is contained in $S'$ for a certain $i$
then all $f(X_{i})\cap f(Y)$ are coincident with $f(Y)\cap S'$ and 
the latter subspace is $(n'-5)$-dimensional.
Since the distance between $f(Y)$ and every $f(X_{i})$ is equal to $2$,
for every $i$ there exists a point 
$$p_{i}\in f(X_{i})\setminus f(Y)$$
collinear to all points of $f(Y)$ (Lemma \ref{lemma-dist}).
If one of the points $p_{i}$ belongs to $S'$ 
then it is contained in every element of $[S',U']_{n'-3}$.
This implies \eqref{eq2-2}, since $f(Y)\cap S'$ is $(n'-5)$-dimensional (Lemma \ref{lemma-dist}).
Suppose that $p_{i}\in f(X_{i})\setminus S'$ for every $i$
and consider the $2$-dimensional singular subspace $T$ spanned by $p_{1},p_{2},p_{3}$.
Every point of $T$ is collinear to all points of $f(Y)$ and $T\cap S'=\emptyset$. 
Each $X'\in[S',U']_{n'-3}$ has a non-empty intersection with $T$, 
i.e. there is a point of $X'\setminus S'$ collinear to all points of $f(Y)$.
As above, we get \eqref{eq2-2}.

Consider the case when every $f(X_{i})\cap f(Y)$ is not contained in $S'$.
The singular subspace $f(Y)\cap S'$ is $(n'-6)$-dimensional. 
We choose three points
$$p_{i}\in (f(X_{i})\cap f(Y))\setminus S',\;\;\;i=1,2,3.$$
The singular subspace spanned by $f(Y)\cap S'$ and $p_{1},p_{2},p_{3}$
is $(n'-3)$-dimensional. Hence it coincides with $f(Y)$.
This means that $f(Y)$ is contained in $U'$. 
Since the dimension of the intersection of two $(n'-3)$-dimensional subspaces of $U'$
is not less than $n'-5$, we have \eqref{eq2-2}.

\section{Proof of Theorem \ref{theorem-main3}}
Suppose that $n=n'$ and $f$ is an isometric embedding of $\Gamma_{n-2}(\Pi)$
in $\Gamma_{n-2}(\Pi')$.
Maximal cliques of the graphs $\Gamma_{n-2}(\Pi)$ and $\Gamma_{n-2}(\Pi')$
are tops and $f$ transfers every top to a subset contained in a top, i.e.
there exists a mapping 
$$g:{\mathcal G}_{n-1}(\Pi)\to {\mathcal G}_{n-1}(\Pi')$$ 
such that 
$$f(\langle U]_{n-2})\subset \langle g(U)]_{n-2}$$
for every $U\in {\mathcal G}_{n-1}(\Pi)$.
The mapping $g$ is injective
(distinct maximal cliques go to subsets of distinct maximal cliques).
 
Two distinct tops of ${\mathcal G}_{n-2}(\Pi)$ have a non-empty intersection 
if and only if the corresponding elements of ${\mathcal G}_{n-1}(\Pi)$
are adjacent vertices of $\Gamma_{n-1}(\Pi)$.
The same holds for tops of ${\mathcal G}_{n-2}(\Pi')$.
This implies that $g$ sends adjacent vertices of $\Gamma_{n-1}(\Pi)$
to adjacent vertices of $\Gamma_{n-1}(\Pi')$. 
Therefore,
$$d(X,Y)\ge d(g(X),g(Y))$$
for all $X,Y\in {\mathcal G}_{n-1}(\Pi)$.
We want to shows that $g$ is an isometric embedding of $\Gamma_{n-1}(\Pi)$
in $\Gamma_{n-1}(\Pi')$. 
Since every geodesic in $\Gamma_{n-1}(\Pi)$ 
can be extended to a maximal geodesic consisting of $n$ edges, 
it is sufficient to establish that
for any $X,Y\in {\mathcal G}_{n-1}(\Pi)$ satisfying $d(X,Y)=n$ 
we have 
$$d(f(X),f(Y))=n.$$

Let $X,Y\in {\mathcal G}_{n-1}(\Pi)$ and $d(X,Y)=n$.
Consider a frame $p_{1},\dots,p_{n},q_{1},\dots,q_{n}$
such that every $p_{i}$ is non-collinear to $q_{i}$ and 
$$X=\langle p_{1},\dots,p_{n}  \rangle,\;\; Y=\langle q_{1},\dots,q_{n}  \rangle.$$
Then 
$$A:=\langle p_{1},\dots,p_{n-1}  \rangle,\;\; B:=\langle p_{2},\dots,p_{n} \rangle$$
and
$$C:=\langle q_{1},\dots,q_{n-1}  \rangle,\;\; D:=\langle q_{2},\dots,q_{n} \rangle$$
are elements of ${\mathcal G}_{n-2}(\Pi)$
satisfying the following conditions:
\begin{enumerate}
\item[(1)] $X=\langle A,B\rangle$ and $Y= \langle C,D\rangle$,
\item[(2)] $d(A,C)=d(B,D)=n$ and $d(A,D)=d(B,C)=n-1$.
\end{enumerate}
We have
$$g(X)=\langle f(A),f(B)\rangle,\;\;g(Y)= \langle f(C),f(D)\rangle$$
and 
$$d(f(A),f(C))=d(f(B),f(D))=n,$$
$$d(f(A),f(D))=d(f(B),f(C))=n-1.$$
Suppose that the distance between $g(X)$ and $g(Y)$ is less than $n$.
Then $g(X)$ and $g(Y)$ have a non-empty intersection.

\begin{lemma}\label{lemma3-1}
Each point of $g(X)\cap g(Y)$ does not belong 
to $f(A)\cup f(B)\cup f(C) \cup f(D)$.
\end{lemma}

\begin{proof}
Since the distance between $f(A)$ and $f(C)$ is equal to $n$,
Lemma \ref{lemma-dist} shows that $f(A)\cap f(C)=\emptyset$ and every point of $f(A)$
is non-collinear to a certain point of $f(C)$.
On the other hand, if $p\in g(X)\cap g(Y)$ then it is collinear to all points of $g(C)$. 
Hence $p$ does not belong to $f(A)$.
Similarly, we show that $p$ does not belong to any of $f(B),f(C),f(D)$.
\end{proof}

The distance between $f(A)$ and $f(D)$ is equal to $n-1$ and,
by Lemma \ref{lemma-dist}, we have two possibilities.
In each of these cases, there is a point $q\in f(D)$ collinear to all points of $f(A)$.
Let $p\in g(X)\cap g(Y)$. By Lemma \ref{lemma3-1}, 
$g(X)$ is spanned by $f(A)$ and $p$. 
Since $q$ is collinear to $p$ and all points of $f(A)$, 
it is collinear to all points of $g(X)$.
The latter means that $q\in g(X)$ (recall that $g(X)$ is a maximal singular subspace of $\Pi'$).
Then $q$ belongs to $g(X)\cap g(Y)$ which contradicts Lemma \ref{lemma3-1}.

Thus $g(X)\cap g(Y)=\emptyset$, i.e. the distance between 
$g(X)$ and $g(Y)$ is equal to $n$. 
So, $g$ is an isometric embedding of $\Gamma_{n-1}(\Pi)$ in $\Gamma_{n-1}(\Pi')$.
By  \cite{Pankov1}, it is induced by 
a collinearity preserving injection of $\Pi$ to $\Pi'$. 
If $S\in {\mathcal G}_{n-2}(\Pi)$ is the intersection of $X,Y\in {\mathcal G}_{n-1}(\Pi)$
then $f(S)$ coincides with $g(X)\cap g(Y)$.
This implies that $f$ is induced by the same collinearity preserving injection of $\Pi$ to $\Pi'$.

\end{document}